\documentclass[11pt,a4paper,reqno]{amsart}
\usepackage{fullpage}
\usepackage{hyperref}
\usepackage{url}

\numberwithin{equation}{section}

\def\Cay{\mathrm{Cay}}

{\theoremstyle{plain}
\newtheorem{theorem}{Theorem}[section]
\newtheorem{lemma}[theorem]{Lemma}
\newtheorem{proposition}[theorem]{Proposition}
\newtheorem{corollary}[theorem]{Corollary}}

{\theoremstyle{remark}   
\newtheorem{remark}[theorem]{Remark}}

\title{On Cayley graphs over generalized dicyclic groups}

\author{Angelot Behajaina}
\email{angelot.behajaina@universite-paris-saclay.fr}
\address{Universit\'e Paris-Saclay, CNRS, Laboratoire de Math\'ematiques d'Orsay, 91405 Orsay, France}

\author{Fran\c cois Legrand}
\email{francois.legrand@unicaen.fr}
\address{Normandie Univ., UNICAEN, CNRS, Laboratoire de Math\'ematiques Nicolas Oresme, 14000 Caen, France}

\begin{document}

\maketitle

\vspace{-1cm}

\begin{abstract}
Recently, several works by a number of authors have studied integrality, distance integrality, and distance powers of Cayley graphs over some finite groups, such as dicyclic groups and (generalized) dihedral groups. Our aim is to generalize and/or to give analogues of these results for generalized dicyclic groups. For example, we give a necessary and sufficient condition for a Cayley graph over a generalized dicyclic group to be integral (i.e., all eigenvalues of its adjacency matrix are in $\mathbb{Z}$). We also obtain sufficient conditions for the integrality of all distance powers of a Cayley graph over a given generalized dicyclic group. These results extend works on dicyclic groups by Cheng--Feng--Huang and Cheng--Feng--Liu--Lu--Stevanovic, respectively.
\end{abstract}

\vspace{-0.2cm}

\section{Introduction} \label{sec:intro}

Throughout this paper, we only consider simple graphs (see \S\ref{sec:prelim} for the basic terminology used in the sequel). A graph is {\it{integral}} if the eigenvalues of its adjacency matrix are all in $\mathbb{Z}$. This notion was introduced by Harary--Schwenk in \cite{HS74}, who proposed the problem of classifying all integral graphs. In \cite{AABS09}, Ahmady--Alon--Blake--Shparlinski showed that integral graphs are ``rare" and, due to this result, it then seems reachable to classify all integral graphs. However, this problem turns out to be difficult and it is then reasonable to start with the study of the integrality of special families of graphs.

The integrality of Cayley graphs over finite groups has been studied by many authors (for the integrality of other kinds of graphs, see, e.g., \cite{BC76, Wat79, WS79}). Given a finite group $G$ and a subset $S$ of $G$ with $1 \not \in S$ and $S^{-1} = S$, the {\it{Cayley graph}} $\Cay(G,S)$ over $G$ with respect to $S$ is the graph whose vertices are the elements of $G$ and whose edges are the 2-sets $\{g,h\}$ with $g^{-1} h \in S$. For $G$ abelian, several characterizations of the integrality of $\Cay(G,S)$ were obtained by Bridges--Mena \cite{BM82}, Klotz--Sander \cite{KS10}, and Alperin--Peterson \cite{AP12}. Moreover, Lu--Huang--Huang focused on the case where $G$ is dihedral in \cite{LHH18}, and their results were extended to generalized dihedral groups by Huang--Li in \cite{HL21}. Furthermore, Cheng--Feng--Huang considered in \cite{CFH19} the case where $G$ is a dicyclic group.

Generalizations of the integrality of Cayley graphs include the integrality of distance powers of Cayley graphs. Given a connected graph $\Gamma$ and a non-empty finite set $D$ of positive integers, the {\it{distance power}} $\Gamma^D$ of $\Gamma$ is the graph whose vertices are those of $\Gamma$ and whose edges are the 2-sets $\{v,w\}$ such that the distance $d_\Gamma(v,w)$ is in $D$. Note that we always have $\Gamma^{\{1\}} = \Gamma$. The integrality of distance powers of Cayley graphs over finite groups $G$ was consi\-dered by Klotz--Sander in \cite{KS12}, who proved that, for $G$ abelian, if a Cayley graph $\Gamma$ over $G$ is integral, then all distance powers $\Gamma^D$ of $\Gamma$ are integral Cayley graphs over $G$. A similar result was obtained by Cheng--Feng--Liu--Lu--Stevanovic in \cite{CFLLS20} for $G$ dihedral and $G$ dicyclic.
 
Unlike integrality, not much work was done on the distance integrality of Cayley graphs (recall that a connected graph is {\it{dis\-tance integral}} if the eigenvalues of its distance matrix are all in $\mathbb{Z}$), maybe because the first general statement describing the eigenvalues of the distance matrix of a given Cayley graph has been made available in the literature only recently (see Huang--Li \cite{HL21}). Using this result, the authors gave there necessary and sufficient conditions on $S$ for ${\rm{Cay}}(G,S)$ to be distance integral in the case where $G$ is a generalized dihedral group, and they obtained a similar result for dicyclic groups in \cite{HL21b}. See also \cite{Ren11, FGK16} for earlier works over specific finite groups, with a geometric flavour. 

Connections between the integrality and the distance integrality of Cayley graphs can also be found in the literature. For example, Ili\'c \cite{Ili10} showed that integral Cayley graphs over cyclic groups are distance integral, and this was extended to abelian groups by Klotz--Sander in \cite{KS12}. In fact, Cayley graphs over abelian groups are integral if and only if they are distance integral, as shown by Huang--Li \cite{HL21b}. Moreover, the same authors \cite{HL21} got sufficient conditions for the equivalence between the integrality and the distance integrality of Cayley graphs over generalized dihedral groups. Although it is unclear from the definitions, these results show that there are links between the integrality and the distance integrality of Cayley graphs.

In the present paper, we give generalizations and/or analogues for generalized dicyclic groups of some of the works quoted above. Given a finite abelian group $A$ of even order and exponent at least 3, and given an element $y$ of $A$ of order 2, recall that the generalized dicyclic group ${\rm{Dic}}(A,y)$ is given by the  presentation $\langle A,x \, | \, x^2=y, xax^{-1}=a^{-1} (a \in A) \rangle.$ In the case where $A$ is cyclic of even order $2n$ with $n \geq 2$ (and $y$ is the unique element of $A$ of order 2), the group ${\rm{Dic}}(A,y)$ is nothing but the dicyclic group ${\rm{Dic}}_n$ with $4n$ elements. More precisely:

\vspace{1.5mm}

\noindent
(1) Given a generalized dicyclic group ${\rm{Dic}}(A,y)$, we obtain a necessary and sufficient condition on a given subset $S$ of ${\rm{Dic}}(A,y)$ with $1 \not \in S$ and $S^{-1} = S$ for $\Cay({\rm{Dic}}(A,y),S)$ to be integral (see Theorem \ref{thm:main_1}). The proof uses the fundamental criteria of Babai \cite{Bab79} and Alperin--Peterson \cite{AP12}, which characterize integral Cayley graphs in terms of the representations and the Boolean algebra of the underlying finite group, respectively, and the classification of all inequivalent irreducible representations of ${\rm{Dic}}(A,y)$, that we establish (see \S\ref{sssec:main_1.2.1}). Our result is an analogue for generalized dicyclic groups of \cite[Theorem 3.1]{HL21} on generalized dihedral groups and, in the case where $A$ is cyclic, we retrieve \cite[Theorem 4.3]{CFH19} on dicyclic groups (see Corollary \ref{coro:main_1.1}). See also Corollary \ref{coro:main_1.2} for another special instance of our result.

\vspace{1.5mm}

\noindent
(2) We provide a sufficient condition on a Cayley graph over a generalized dicyclic group for its distance powers to be integral Cayley graphs (see Theorem \ref{thm:main_2}). Our result extends \cite[Theorem 4.4]{CFLLS20} on dicyclic groups. See also Corollaries \ref{coro:main_2.1} and \ref{coro:main_2.2} for other special cases of our result, the second of which yields infinitely many connected Cayley graphs over generalized dicyclic groups whose all distance powers are integral Cayley graphs.

\vspace{1.5mm}

\noindent
(3) We characterize distance integral Cayley graphs over generalized dicyclic groups (see Theorem \ref{thm:main_3}). To that end, we follow the method developed in \cite{HL21} but the proof in the dicyclic case is more subtle because of the one-dimensional representations of the underlying abelian group whose kernel does not contain the chosen element of order 2. Our result is an analogue for generalized dicyclic groups of \cite[Theorem 4.3]{HL21} on generalized dihedral groups and, in the case where the underlying abelian group is cyclic, we retrieve \cite[Theorem 1.3]{HL21b} on dicyclic groups. See also Corollaries \ref{coro:main_3.1} and \ref{coro:main_3.2} for other special instances of Theorem \ref{thm:main_3}.

\vspace{1.5mm}

\noindent
(4) We contribute to establishing links between integrality and distance integrality of Cayley graphs, by providing a simple sufficient condition for the equivalence between these two notions over generalized dicyclic groups (see Theorem \ref{thm:main_4}). Our result is an analogue for generalized dicyclic groups of \cite[Theorem 5.2]{HL21} on generalized dihedral groups and we point out that our result is new already for dicyclic groups. See also Corollary \ref{coro:main_4.1} for a combination of our results yielding infinitely many distance integral connected Cayley graphs over generalized dicyclic groups whose connected distance powers are distance integral Cayley graphs.

\vspace{1.5mm}

To conclude this introduction, let us mention that, for the sake of shortness, we have not systematically considered all possible generalizations and/or analogues for generalized dicyclic groups of the results from the recent works quoted above. We refer to Remark \ref{rk:2} for a few examples of potential results that we leave to the interested reader.

\section{Preliminaries} \label{sec:prelim}

In this section, we collect the material on representations of finite groups, graphs, and ge\-ne\-ra\-lized dicyclic groups that will be used in the sequel. For a subset $S$ of a finite group $G$, we set ${S}^{-1} = \{s^{-1} \, : \, s \in S\}$ and, for $S_1, S_2 \subseteq G$, we set $S_1 S_2 = \{s_1s_2 \, : s_1 \in S_1, s_2 \in S_2\}$. Given a positive integer $n$, we set $S^{(n)} = \{s_1 \cdots s_n \, : \, s_1, \dots, s_n \in S\}$, and $S^{(0)} = \{1\}$.

\subsection{Representations of finite groups} \label{ssec:prelim_1}

A (complex) {\it{representation}} of a given finite group $G$ is a group homomorphism $\rho : G \rightarrow {\rm{GL}}(V)$, where $V$ is a finite dimensional complex vector space. The {\it{dimension}} ${\rm{dim}}(\rho)$ of $\rho$ is the dimension of $V$, and the {\it{character}} of $\rho$ is the map $\chi_\rho : G \rightarrow \mathbb{C}$ defined by $\chi_\rho(g) = {\rm{Tr}}(\rho(g))$ for $g \in G$. The representation $\rho$ is {\it{unitary}} with respect to a given Hermitian inner product $(\cdot | \cdot)$ of $V$ \footnote{Throughout the paper, we assume $(\lambda v | w) = \lambda (v | w)$ and $(v | \lambda w) = \overline{ \lambda} (v | w)$ for $\lambda \in \mathbb{C}$ and $v,w \in V$.} if $(\rho(g)v | \rho(g)w) = (v | w)$ for $g \in G$ and $v,w \in V$. It is well-known that $\rho$ is {\it{unitarisable}}, i.e., there is a  Hermitian inner product of $V$ (depending on $\rho$) for which $\rho$ is unitary.

A subspace $W$ of $V$ is {\it{invariant}} under $\rho$ if $\rho(g)(W) = W$ for $g \in G$, and $\rho$ is {\it{irreducible}} if $V \not=\{0\}$ and if $\{0\},V$ are the only invariant subspaces of $V$. Letting $L^2(G)$ be the vector space of all functions $f : G \rightarrow \mathbb{C}$, together with the Hermitian inner product $( \cdot | \cdot)$ given by
\begin{equation} \label{eq:scalar}
(f_1 | f_2) = \frac{1}{|G|} \sum_{g \in G} f_1(g) \overline{f_2(g)}
\end{equation}
for $f_1, f_2 \in L^2(G)$, the representation $\rho$ is irreducible if and only if $(\chi_\rho| \chi_\rho)=1$. 

For a finite set $S$ and a function $f : S \rightarrow \mathbb{C}$, we set $f(S) = \sum_{s \in S} f(s)$. We shall also say that a subset $S$ of $G$ is {\it{integral}} if $\chi_\rho(S) \in \mathbb{Z}$ for every irreducible representation $\rho$ of $G$.

Two representations $\rho_1 : G \rightarrow {\rm{GL}}(V_1)$ and $\rho_2 : G \rightarrow {\rm{GL}}(V_2)$ of $G$ are {\it{equivalent}} if there is an isomorphism $T : V_1 \rightarrow V_2$ such that, for every $g \in G$, we have $T \circ \rho_1(g) = \rho_2(g) \circ T$. Classically, $\rho_1$ and $\rho_2$ are equivalent if and only if $\chi_{\rho_1} = \chi_{\rho_2}$. The group $G$ has only finitely many inequivalent irreducible representations, and their number equals the number of conjugacy classes of $G$. Moreover, if $\rho_1, \dots, \rho_n$ are the inequivalent irreducible representations of $G$, then 
\begin{equation} \label{eq:squares}
|G| = ({\rm{dim}}(\rho_1))^2 + \cdots + ({\rm{dim}}(\rho_n))^2.  
\end{equation}
In particular, if $G$ is abelian, then every irreducible representation of $G$ has dimension 1.

\subsection{Graphs} \label{ssec:prelim_2} 

A (simple) {\it{graph}} is a couple $\Gamma = (V,E)$ where $V$ is a non-empty finite set and $E$ is a finite set whose elements are 2-sets $\{v,w\}$ with $v,w \in V$. The elements of $V$ (resp., of $E$) are the {\it{vertices}} (resp., {\it{edges}}) of $\Gamma$, and two vertices $v,w$ are {\it{adjacent}} if $\{v,w\} \in E$. Choosing a bijection $V \rightarrow \{1, \dots, |V|\}$, the {\it{adjacency matrix}} $A(\Gamma)$ of $\Gamma$ is the $|V| \times |V|$-matrix whose $(i,j)$-entry equals 1 if $v_i$ and $v_j$ are adjacent, and 0 otherwise ($1 \leq i,j \leq |V|$). A graph $\Gamma$ is {\it{integral}} if all eigenvalues of $A(\Gamma)$, which are always real numbers, are actually in $\mathbb{Z}$.

The graph $\Gamma$ is {\it{connected}} if, given $v \not= w \in V$, there exist $v_0, v_1, \dots, v_n \in V$ such that $v = v_0$, $v_i$ and $v_{i+1}$ are adjacent for $i=0, \dots, n-1$, and $w=v_n$. If $\Gamma$ is connected, the {\it{distance}} $d_\Gamma(v,w)$ is the least $n \geq 1$ as in the previous sentence and, choosing a bijection $V \rightarrow \{1, \dots, |V|\}$, the {\it{distance matrix}} $D(\Gamma)$ of $\Gamma$ is the $|V| \times |V|$-matrix whose $(i,j)$-entry equals $d_\Gamma(v_i, v_j)$ ($1 \leq i,j \leq |V|$). A connected graph $\Gamma$ is {\it{distance integral}} if all eigenvalues of $D(\Gamma)$, which are always real numbers, are actually in $\mathbb{Z}$.

Let $\Gamma$ be a connected graph and $D$ a non-empty finite set of positive integers. The {\it{distance power}} $\Gamma^D$ of $\Gamma$ is the graph whose vertices are exactly those of $\Gamma$ and such that two distinct vertices $v, w$ are adjacent if $d_\Gamma(v,w) \in D$.

Let $G$ be a finite group and $S \subseteq G$ with $1 \not \in S$ and $S^{-1} = S$. The {\it{Cayley graph}} $\Cay(G,S)$ over $G$ with respect to $S$ is the graph whose vertices are the elements of $G$ and whose edges are the 2-sets $\{g,h\}$ with $g^{-1} h \in S$. Note that $\Cay(G,S)$ is connected if and only if $\langle S \rangle = G$.

In \S\ref{sec:main_1}, we shall be interested in the integrality of Cayley graphs and, hence, shall make use of the following result, due to Babai (see \cite{Bab79}):

\begin{lemma} \label{lemma:babai}
Let $G$ be a finite group and $S \subseteq G$ with $1 \not \in S$ and $S^{-1} = S$. Then $\Cay(G,S)$ is integral if and only if the eigenvalues of $\sum_{s \in S} \rho(s)$ are in $\mathbb{Z}$ for every irreducible representation $\rho$ of $G$.
\end{lemma}

By the last lemma, a necessary condition for ${\rm{Cay}}(G,S)$ to be integral is that $S$ is integral (and the converse holds if $G$ is abelian). Given a finite abelian group $G$, we shall also need the characterization of integral subsets of $G$ in terms of the Boolean algebra of $G$. To that end, we let $\mathcal{F}_G$ denote the set of all subgroups of $G$. The {\it{Boolean algebra}} $B(G)$ is the set whose elements are obtained by arbitrary finite intersections, unions, and complements of elements of $\mathcal{F}_G$. The next lemma has been established by Alperin--Peterson in \cite{AP12}:

\begin{lemma} \label{lemma:ap}
Let $G$ be a finite abelian group. A subset $S$ of $G$ is integral if and only if $S \in B(G)$.
\end{lemma}

In view of our study of distance powers of Cayley graphs in \S\ref{sec:main_2}, let us recall that, for a finite abelian group $G$, the minimal elements of $B(G)$ are called {\it{atoms}}. Let $\widetilde{B}(G)$ be the set of all atoms of $G$. As proved by Alperin--Peterson (see \cite{AP12}), every element of $B(G)$ is the union of some atoms of $G$, and every atom is of the form $[g] = \{ x \in G \, : \, \langle x \rangle = \langle g \rangle \}$ ($g \in G$). In particular, we have:

\begin{lemma} \label{lemma:ap_0}
Let $G$ be a finite abelian group and $S \in B(G)$. Then $S = S^{-1}$.
\end{lemma}

Moreover, the Boolean algebra $B(G)$ of $G$ has this property (see \cite{KS12} for a proof):

\begin{lemma} \label{lemma:ks}
Let $G$ be a finite abelian group and $S,T \in B(G)$. Then $ST \in B(G)$.
\end{lemma}

In \S\ref{sec:main_3}, we shall make use of the distance integrality criterion below. To that end, let $G$ be a finite group and $S$ a subset of $G$ with $1 \not \in S$, $S^{-1} = S$, and $\langle S \rangle = G$. We let \begin{equation} \label{eq:l_s}
\ell_S 
\end{equation}
denote the function $G \rightarrow \mathbb{N}$ defined by $\ell_S(g) = k$ if $g \not=1$, where $k$ is the least $n \geq 1$ for which there exist $x_1, \dots, x_n \in S$ with $g= x_1 \cdots x_n$, and by $\ell_S(g) = 0$ if $g=1$. Moreover, given an irreducible representation $\rho : G \rightarrow {\rm{GL}}(V)$ and a Hermitian inner product $( \cdot | \cdot)$ of $V$ for which $\rho$ is unitary, let $v_1, \dots, v_{{\rm{dim}}(\rho)}$ denote an orthonormal basis of $V$. For $1 \leq i, j \leq {\rm{dim}}(\rho)$, we let 
\begin{equation} \label{eq:varphi}
\varphi_{i,j}^\rho
\end{equation}
denote the function $G \rightarrow \mathbb{C}$ defined by $\varphi_{i,j}^\rho(g) = (\rho(g)(v_j) |v_i).$ Let
\begin{equation} \label{eq:Phi}
\Phi_S(\rho) 
\end{equation}
denote the ${\rm{dim}}(\rho) \times {\rm{dim}}(\rho)$-matrix whose $(i,j)$-entry equals 
$$\sum_{g \in G} \ell_S(g) \varphi_{i,j}^\rho(g)$$ 
for $1 \leq i,j \leq {\rm{dim}}(\rho)$. The next lemma is an immediate consequence of \cite[Theorem 4.2]{HL21}:

\begin{lemma} \label{lemma:HL}
Let $G$ be a finite group and let $S$ be a subset of $G$ with $1 \not \in S$, $S^{-1} = S$, and $\langle S \rangle = G$. Let $\rho_1, \dots, \rho_n$ be the inequivalent irreducible representations of $G$. Then ${\rm{Cay}}(G,S)$ is distance integral if and only if the eigenvalues of $\Phi_S(\rho_1), \dots, \Phi_S(\rho_n)$ are all in $\mathbb{Z}$.
\end{lemma}

Finally, a {\it{multi-set}} is a couple $(S,f)$, where $S$ is a set and $f$ a function from $S$ to $\mathbb{N}$. For $S$ finite and $\kappa : S \rightarrow \mathbb{C}$, we set $\kappa(S,f) = \sum_{s \in S} f(s) \kappa(s)$. Given two multi-sets $(S_1,f_1)$ and $(S_2,f_2)$, the {\it{sum}} $(S_1,f_1) + (S_2,f_2)$ is the multi-set $(S,f)$ with $S = S_1 \cup S_2$ and $f(s) = \widetilde{f_1}(s) + \widetilde{f_2}(s)$ for $s \in S$, where $\widetilde{f_i}(s) = f_i(s)$ if $s \in S_i$ and $\widetilde{f_i}(s) = 0$ if $s \not \in S_i$ ($i=1,2$). 

Given a finite group $G$, we say that a multi-set $(S,f)$ with $S \subseteq G$ is {\it{integral}} if $\chi_\rho(S,f) \in \mathbb{Z}$ for every irreducible representation $\rho$ of $G$. If $G$ is abelian, we also let $C(G)$ be the set whose elements are of the form $\sum_{a \in T} (a, f_a)$ with $T \subseteq \widetilde{B}(G)$ and $f_a$ constant. The next lemma (see, e.g., \cite[Lemma 3.1]{LHH18}), which extends Lemma \ref{lemma:ap} to multi-sets, will also be needed in \S\ref{sec:main_3}:

\begin{lemma} \label{lemma:multi}
Let $G$ be a finite abelian group and $(S,f)$ a multi-set with $S \subseteq G$. Then $(S,f)$ is integral if and only if $(S,f) \in C(G)$.
\end{lemma}

\subsection{Generalized dicyclic groups} \label{ssec:prelim_3}

For the rest of the paper, let $A$ be a finite abelian group with even order and exponent at least 3, and let $y$ be an element of $A$ of order 2. The generalized dicyclic group ${\rm{Dic}}(A,y)$ is given by the following presentation:
$${\rm{Dic}}(A,y) = \langle A,x \, | \, x^2=y, xax^{-1}=a^{-1} (a \in A) \rangle.$$ 
Note that ${\rm{Dic}}(A,y) = A \coprod xA$ and that $|{\rm{Dic}}(A,y)|=2|A|$. Given a subset $S$ of ${\rm{Dic}}(A,y)$ with $1 \not \in S$, $S^{-1} = S$, and $\langle S \rangle = {\rm{Dic}}(A,y)$, we let 
\begin{equation} \label{eq:l_s_point}
\ell_S(x \cdot)
\end{equation}
denote the function $A \rightarrow \mathbb{N}$ defined by $\ell_S(x \cdot)(a) = \ell_S(xa)$ ($a \in A$), where $\ell_S$ is defined in \eqref{eq:l_s}.

\section{Integrality} \label{sec:main_1}

\subsection{Main result} \label{ssec:main_1.1}

The aim of this section is the next theorem, which characterizes integral Cayley graphs over generalized dicyclic groups:

\begin{theorem} \label{thm:main_1}
Let $S$ be a subset of ${\rm{Dic}}(A,y)$ such that $1 \notin S$ and $S^{-1}=S$. Set $S= S_1 \cup xS_2$ with $S_1,S_2 \subseteq A$. Then $\Cay({\rm{Dic}}(A,y),S)$ is integral if and only if the next two conditions hold:

\noindent
{\rm{(1)}} $S_1 \in B(A)$,

\noindent
{\rm{(2)}} for every one-dimensional representation $\pi : A \rightarrow \mathbb{C}^*$ of $A$ with $A^2 \not \subseteq {\rm{ker}}(\pi)$, there exists $\alpha \in \mathbb{Z}$ such that $\pi(S_2) \pi(S_2^{-1})= \alpha^2$.
\end{theorem}

\subsection{Proof of Theorem \ref{thm:main_1}} \label{ssec:main_1.2}

Similarly to the proofs given in \cite{LHH18, CFH19, HL21}, we shall make use of Lemmas \ref{lemma:babai} and \ref{lemma:ap}, the first one of which involves the irreducible representations of the underlying group. Therefore, we first determine all inequivalent irreducible representations of ${\rm{Dic}}(A,y)$ (\S\ref{sssec:main_1.2.1}) and then proceed to the proof of Theorem \ref{thm:main_1} (\S\ref{sssec:main_1.2.2}). Note that the next classification will be reused to prove our results on distance integrality in \S\ref{sec:main_3}.

\subsubsection{Irreducible representations of ${\rm{Dic}}(A,y)$} \label{sssec:main_1.2.1}

First, note that $A/A^2$ (with $A^2 = \{a^2 \, : \, a \in A\}$) is an elementary abelian 2-group. Set $|A/A^2| = 2^n$ with $n \geq 0$.

\vskip 1mm

\noindent
(1) \textit{One-dimensional representations.} Let $\rho : {\rm{Dic}}(A,y) \rightarrow \mathbb{C}^*$ be a one-dimensional representation of ${\rm{Dic}}(A,y)$. For $a \in A$, we have $\rho(a)=\rho(x)\rho(a)\rho(x)^{-1}=\rho(xax^{-1})=\rho(a^{-1})$ 
and so $\rho(a^2)=1$, i.e., $\rho$ is trivial on $A^2$. Therefore, every one-dimensional representation of ${\rm{Dic}}(A,y)$ gives rise to a one-dimensional representation of $A/A^2$. 

Conversely, let $\pi:A/A^2\rightarrow \mathbb{C}^*$ be a one-dimensional representation of $A/A^2$. Then there are exactly two inequivalent one-dimensional representations $\rho_{1}, \rho_2$ of ${\rm{Dic}}(A,y)$ giving rise to $\pi$ and, denoting the reduction modulo $A^2$ of $a \in A$ by $\overline{a}$, they are actually given by

\noindent
- $\rho_1(x) = 1$ (resp., $\rho_2(x) = -1$) and $\rho_1(a)=\pi(\overline{a})$ (resp., $\rho_2(a)=\pi(\overline{a})$) for $a \in A$ (if $\pi(\overline{y})=1$),

\noindent
- $\rho_1(x) = i$ (resp., $\rho_2(x) = -i$) and $\rho_1(a)=\pi(\overline{a})$ (resp., $\rho_2(a)=\pi(\overline{a})$) for $a \in A$ (if $\pi(\overline{y})=-1$).

Hence, ${\rm{Dic}}(A,y)$ has exactly $2^{n+1}$ inequivalent one-dimensional representations.

\vskip 1mm

\noindent
(2) \textit{Two-dimensional representations.} Let $\pi : A \rightarrow \mathbb{C}^*$ be a one-dimensional re\-pre\-sen\-tation of $A$ with $A^2 \not \subseteq {\rm{ker}}(\pi)$. Let $R_\pi : {\rm{Dic}}(A,y) \rightarrow {\rm{GL}}_2(\mathbb{C})$ be the representation of ${\rm{Dic}}(A,y)$ given by
\begin{equation}\label{equmatrepind}
R_{\pi}(x)=\begin{pmatrix} 0 & \pi(y) \\ 1 & 0 \end{pmatrix} \quad {\rm{and}} \quad R_{\pi}(a)= \begin{pmatrix} \pi(a) & 0 \\ 0 & \pi(a^{-1}) \end{pmatrix}
\end{equation}
for $a \in A$. In fact, $R_\pi$ is the representation $\mathrm{Ind}_{A}^{{\rm{Dic}}(A,y)}(\pi)$ of ${\rm{Dic}}(A,y)$ induced by $\pi$. 

\begin{lemma} \label{lem:irreducible}
The representation $R_\pi$ is irreducible.
\end{lemma}

\begin{proof}
As recalled in \S\ref{ssec:prelim_1}, it suffices to show 
\begin{equation} \label{eq:0}
(\chi_{R_\pi} | \chi_{R_\pi})=1,
\end{equation}
where $(\cdot | \cdot)$ is defined in \eqref{eq:scalar}. To that end, note that, for $a \in A$, we have
$$R_\pi(xa) = \begin{pmatrix} 0 & \pi(y) \\ 1 & 0 \end{pmatrix} \begin{pmatrix} \pi(a) & 0 \\ 0 & \pi(a^{-1}) \end{pmatrix} = \begin{pmatrix} 0 & \pi(ya^{-1}) \\ \pi(a) & 0 \end{pmatrix}.$$
Moreover, as $\pi(a)$ is a root of unity for $a \in A$, we have $\pi(a) + \pi(a^{-1}) \in \mathbb{R}$ for $a \in A$. Hence,
\begin{equation} \label{eq:1}
(\chi_{R_\pi} | \chi_{R_\pi}) = \frac{1}{2|A|} \sum_{a \in A} (\pi(a) + \pi(a^{-1}))^2  = 1 + \frac{1}{|A|} \sum_{a\in A} \pi(a^2).
\end{equation}
Now, since $A^2 \not \subseteq {\rm{ker}}(\pi)$, there exists $b \in A$ with $\pi(b^2) \not=1$. Since
$$\sum_{a\in A} \pi(a^2) = \sum_{a\in A} \pi((ba)^2) = \pi(b^2) \sum_{a\in A} \pi(a^2),$$
we get
\begin{equation} \label{eq:2}
\sum_{a\in A} \pi(a^2)=0.
\end{equation}
It then remains to combine \eqref{eq:1} and \eqref{eq:2} to get \eqref{eq:0}, as needed for the lemma.
\end{proof}

\begin{lemma} \label{lemma:conj}
For $i=1,2$, let $\pi_i : A \rightarrow \mathbb{C}^*$ be a one-dimensional representation of $A$ with $A^2 \not \subseteq {\rm{ker}}(\pi_i)$. Then $R_{\pi_1}$ and $R_{\pi_2}$ are equivalent if and only if $\pi_1= \pi_2$ or $\pi_1 = \overline{\pi_2}$.
\end{lemma}

\begin{proof}
First, assume $\pi_1= \pi_2$ or $\pi_1 = \overline{\pi_2}$. Then $\chi_{R_{\pi_1}} = \chi_{R_{\pi_2}}$ and, hence, $R_{\pi_1}, R_{\pi_2}$ are equivalent. Conversely, if $R_{\pi_1}, R_{\pi_2}$ are equivalent, then $\chi_{R_{\pi_1}} = \chi_{R_{\pi_2}}$ and so $\pi_1(a)+\overline{\pi_1(a)}=\pi_2(a)+\overline{\pi_2(a)}$ for $a \in A$. Hence, for $a \in A$, since $\pi_1(a)$ and $\pi_2(a)$ are roots of unity, either $\pi_1(a)=\pi_2(a)$ or $\pi_1(a)=\overline{\pi_2(a)}$. We now consider a case distinction.

\vspace{1mm}

\noindent
$\bullet$ Suppose there is $a \in A$ with $\pi_1(a) \notin \mathbb{R}$ and $\pi_1(a)=\overline{\pi_2(a)}$. Then $\pi_1 = \overline{\pi_2}$. Indeed, assume there is $b \in A$ with $\pi_1(b) \neq \overline{\pi_2(b)}$, so $\pi_1(b)=\pi_2(b)$. If $\pi_1(ab)=\pi_2(ab)$, then  $$\pi_1(a)\pi_1(b)=\pi_1(ab)=\pi_2(ab)=\pi_2(a)\pi_2(b)=\overline{\pi_1(a)}\pi_2(b)=\overline{\pi_1(a)}\pi_1(b),$$ 
so $\pi_1(a)=\overline{\pi_1(a)}$, a contradiction. If $\pi_1(ab)= \overline{\pi_2(ab)}$, a similar computation leads to $\pi_2(a)\pi_2(b)=\pi_2(a) \overline{\pi_1(b)}$ and so $\pi_2(b)=\overline{\pi_1(b)}$, which cannot happen either. 

\vspace{1mm}

\noindent
$\bullet$ Suppose that, for $a \in A$, either $\pi_1(a) \in \mathbb{R}$ or $\pi_1(a) \neq \overline{\pi_2(a})$. It is then easily checked that $\pi_1(a)=\pi_2(a)$ for $a \in A$. 
\end{proof}

Since there are exactly $|A|-2^n$ inequivalent one-dimensional representations $\pi : A \rightarrow \mathbb{C}^*$ of $A$ with $A^2 \not \subseteq {\rm{ker}}(\pi)$, we then deduce from Lemmas \ref{lem:irreducible} and \ref{lemma:conj} that there are at least $(|A|-2^n)/2$ inequivalent two-dimensional irreducible representations of ${\rm{Dic}}(A,y)$.

\vskip 1.5mm

\noindent
(3) {\it{Conclusion.}} Since $$1^2(2^{n+1})+2^2\left( \frac{|A|-2^n}{2} \right)=2|A|=|{\rm{Dic}}(A,y)|,$$ we get from \eqref{eq:squares} that ${\rm{Dic}}(A,y)$ has exactly $(|A|-2^n)/2$ inequivalent two-dimensional irreducible representations, and that this group has no irreducible representation of dimension at least 3.

\subsubsection{Proof of Theorem \ref{thm:main_1}} \label{sssec:main_1.2.2}

By Lemma \ref{lemma:babai} and \S\ref{sssec:main_1.2.1}, it suffices to show that (1) and (2) in Theorem \ref{thm:main_1} hold if and only if the eigenvalues of $\sum_{s \in S} \rho(s)$ are in $\mathbb {Z}$ for every one-dimensional representation $\rho : {\rm{Dic}}(A,y) \rightarrow \mathbb{C}^*$ of ${\rm{Dic}}(A,y)$ and those of $\sum_{s \in S} R_\pi(s)$ are in $\mathbb {Z}$ for every one-dimensional representation $\pi : A \rightarrow \mathbb{C}^*$ of $A$ with $A^2 \not \subseteq {\rm{ker}}(\pi)$, where $R_\pi$ is defined in \eqref{equmatrepind}.

First, let $\rho : {\rm{Dic}}(A,y) \rightarrow \mathbb{C}^*$ be a one-dimensional representation of ${\rm{Dic}}(A,y)$. Then, by \S\ref{sssec:main_1.2.1}, we have $\sum_{s \in S} \rho(s) \in \mathbb{Z}[i]$ and, since $S = S^{-1}$, we actually have $\sum_{s \in S} \rho(s) \in \mathbb{Z}$.

Now, let $\pi : A \rightarrow \mathbb{C}^*$ be a one-dimensional representation of $A$ with $A^2 \not \subseteq {\rm{ker}}(\pi)$. We have
\begin{align*}
\sum_{s \in S}R_\pi(s)&=\sum_{s \in S_1}R_\pi(s)+\sum_{s \in S_2}R_\pi(xs)\\
                  &=\sum_{s \in S_1} \begin{pmatrix} \pi(s) & 0 \\ 0 & \pi(s^{-1}) \end{pmatrix} + \sum_{s \in S_2} \begin{pmatrix} 0 & \pi(y) \\ 1 & 0 \end{pmatrix} \begin{pmatrix} \pi(s) & 0 \\ 0 & \pi(s^{-1}) \end{pmatrix}\\
                  &= \begin{pmatrix} \pi(S_1) & 0 \\
                  0 & \pi(S_1^{-1})\end{pmatrix} + \sum_{s \in S_2} \begin{pmatrix} 0 & \pi(ys^{-1}) \\ \pi(s) & 0 \end{pmatrix}  \\
                  &=\begin{pmatrix} \pi(S_1) & \pi(yS_2^{-1}) \\
                  \pi(S_2) & \pi(S_1^{-1})\end{pmatrix}.
\end{align*}
Noticing that $(xa)^{-1}=xya$ for $a \in S_2$, we have $S^{-1}=S$ if and only if $S_{1}^{-1}=S_1$ and $yS_{2}=S_2$. Hence, $\pi(S_1^{-1}) = \pi(S_1)$ and $\pi(yS_2^{-1}) = \pi(S_2^{-1})$, thus yielding
$$\sum_{s \in S}R_\pi(s) = \begin{pmatrix} \pi(S_1) & \pi(S_2^{-1}) \\
                  \pi(S_2) & \pi(S_1)\end{pmatrix}.$$
Therefore,
$${\rm det} \left( XI_2-\sum_{s \in S}R_\pi(s) \right)= \begin{vmatrix} X - \pi(S_1) & -\pi(S_2^{-1})\\
-\pi(S_2) & X-\pi(S_1)
\end{vmatrix}
=X^2-2\pi(S_1)X+\pi(S_1)^2-\pi(S_2)\pi(S_2^{-1}).$$
Hence, the eigenvalues of $\sum_{s \in S}R_\pi(s)$ are $\lambda_{1}=\pi(S_1) - \sqrt{\pi(S_2)\pi(S_2^{-1})}$ and $\lambda_{2}=\pi(S_1) + \sqrt{\pi(S_2)\pi(S_2^{-1})}$. If $\lambda_1, \lambda_2$ are in $\mathbb{Z}$, then $\pi(S_1)=(\lambda_1+\lambda_2)/{2}$ is in $\mathbb{Q}$. Since $\pi(S_1)$ is also integral over $\mathbb{Z}$, we get that $\pi(S_1)$ is actually in $\mathbb{Z}$. Moreover, $\sqrt{\pi(S_2)\pi(S_2^{-1})}$ is also in $\mathbb{Z}$. Conversely, if $\pi(S_1)$ is in $\mathbb{Z}$ and $\pi(S_2)\pi(S_2^{-1}) = \alpha^2$ for some $\alpha \in \mathbb{Z}$, then $\lambda_1$ and $\lambda_2$ are in $\mathbb{Z}$. 

Finally, using Lemma \ref{lemma:ap} and the fact that $\pi(S_1)$ is in $\mathbb{Z}$ for every one-dimensional representation $\pi : A \rightarrow \mathbb{C}^*$ of $A$ with $A^2 \subseteq {\rm{ker}}(\pi)$, we get the desired equivalence.

\subsection{Corollaries} \label{ssec:main_1.3}

We first apply Theorem \ref{thm:main_1} for $A$ cyclic, and then retrieve \cite[Theorem 4.3]{CFH19}. To that end, assume $A$ is cyclic of order $2m$ with $m \geq 2$ (and so $y$ is the unique element of order 2 of $A$). Let $\pi_1 , \dots, \pi_{2m-2}$ be the inequivalent one-dimensional representations $\pi : A \rightarrow \mathbb{C}^*$ of $A$ whose kernel does not contain $A^2$. Up to reordering, we may assume $\pi_{m-1 + i} = \overline{\pi_i}$ for $i = 1, \dots, m-1$. Consider the representations $R_{\pi_1}, \dots, R_{\pi_{m-1}}$ of ${\rm{Dic}}(A,y)$ defined in \eqref{equmatrepind}. 
\begin{corollary} \label{coro:main_1.1}
Assume $A$ is cyclic of order $2m$ with $m \geq 2$ (and so $y$ is the unique element of order 2 of $A$). Let $S \subseteq {\rm{Dic}}(A,y)$ be such that $1 \not \in S$ and $S^{-1} = S$. Set $S = S_1 \cup x S_2$ with $S_1, S_2 \subseteq A$. Then $\Cay({\rm{Dic}}(A,y),S)$ is integral if and only if $S_1 \in B(A)$ and, for $i=1, \dots, m-1$, there exists $\alpha_i \in \mathbb{Z}$ such that $2 \chi_{R_{\pi_i}}((xS_2)^{(2)}) = \alpha_i^2$.
\end{corollary}

\begin{proof}
Given $i \in \{1, \dots, m-1\}$, we have $\chi_{R_{\pi_i}}(xsxt) = \pi_i(y)(\pi_i(s) \pi_i(t^{-1}) + \pi_i(s^{-1}) \pi_i(t))$ for $s,t \in S_2$. Hence, 
$$2 \sum_{s,t \in S_2} \chi_{R_{\pi_i}}(xsxt) = 4 \pi_i(y) \pi_i(S_2) \pi_i(S_2^{-1}).$$
But we have $yS_2 = S_2$ (as $S^{-1} = S$) and, hence, $\pi_i(y) \pi_i(S_2) = \pi_i(S_2)$. Therefore, 
$$2 \sum_{s,t \in S_2} \chi_{R_{\pi_i}}(xsxt) = 4 \pi_i(S_2) \pi_i(S_2^{-1})$$ 
and it remains to use Theorem \ref{thm:main_1} to conclude the proof.
\end{proof}

We conclude this part with the next result, which will be extended to distance powers in \S\ref{ssec:main_2.3}:

\begin{corollary} \label{coro:main_1.2}
Let $S \subseteq {\rm{Dic}}(A,y)$ be such that $1 \notin S$ and $S^{-1}=S$. Set $S= S_1 \cup xS_2$ with $S_1,S_2 \subseteq A$. The following two conditions are equivalent:

\vspace{0.5mm}

\noindent
{\rm{(1)}} $S_1, S_2 \in B(A)$,

\vspace{0.5mm}

\noindent
{\rm{(2)}} $\Cay({\rm{Dic}}(A,y),S)$ is integral and $S_2 = S_2^{-1}$.
\end{corollary}

\begin{proof}
First, assume (1) holds. Then $S_2 = S_2^{-1}$ by Lemma  \ref{lemma:ap_0}. Moreover, given a one-di\-men\-sional repre\-sen\-tation $\pi : A \rightarrow \mathbb{C}^*$ of $A$ with $A^2 \not \subseteq {\rm{ker}}(\pi)$, we have $\pi(S_2) \pi(S_2^{-1}) = \pi(S_2)^2$ and, as $S_2 \in B(A)$, we have $\pi(S_2) \in \mathbb{Z}$ by Lemma \ref{lemma:ap}. It then remains to use Theorem \ref{thm:main_1} to get that $\Cay({\rm{Dic}}(A,y),S)$ is integral. Conversely, if (2) holds, then, given a one-dimensional representation $\pi : A \rightarrow \mathbb{C}^*$ of $A$ with $A^2 \not \subseteq {\rm{ker}}(\pi)$, using $S_2^{-1} = S_2$, we have $\pi(S_2) \pi(S_2^{-1}) = \pi(S_2)^2$. It then remains to combine Lemma \ref{lemma:ap}, Theorem \ref{thm:main_1}, and the fact that $\pi(S_2)$ is in $\mathbb{Z}$ for every one-dimensional representation $\pi : A \rightarrow \mathbb{C}^*$ of $A$ with $A^2 \subseteq {\rm{ker}}(\pi)$ to get (1).
\end{proof}

\begin{remark} \label{rk:15}
By Corollary \ref{coro:main_1.2}, we always have ``$S_1, S_2 \in B(A) \Rightarrow$ $\Cay({\rm{Dic}}(A,y),S)$ integral". However, there are integral Cayley graphs $\Cay({\rm{Dic}}(A,y),S)$ which do not fulfill $S_2 = S_2^{-1}$ (see the remark after Corollary 4.5 in \cite{CFH19} for examples, already with $A$ cyclic), thus showing that the converse in the last implication fails in general.
\end{remark}

\section{Distance powers} \label{sec:main_2}

\subsection{Main result} \label{ssec:main_2.1}

Our aim is the next theorem, which extends both the implication from Remark \ref{rk:15} (which assumes $D=\{1\}$) and \cite[Theorem 4.4]{CFLLS20} (which assumes $A$ cyclic):

\begin{theorem} \label{thm:main_2}
Let $S \subseteq {\rm{Dic}}(A,y)$ be such that $1 \not \in S$, $S^{-1} = S$, and $\langle S\rangle = {\rm{Dic}}(A,y)$, and let $D$ be a non-empty finite subset of $\mathbb{N} \setminus \{0\}$. Set $S = S_1 \cup xS_2$ with $S_1, S_2 \subseteq A$. Assume $S_1, S_2 \in B(A)$. Then there are subsets $S_1^D, S_2^D$ of $A$ such that, with $S^D=S_1^D \cup xS_2^D$, we have  

\newpage

\noindent
$1 \not \in S^D$, $S^D = (S^D)^{-1}$, and $\Cay({\rm{Dic}}(A,y),S)^D = \Cay({\rm{Dic}}(A,y),S^D)$. Moreover, $\Cay({\rm{Dic}}(A,y),S)^D$ is {integral} and $S_2^D=(S_2^D)^{-1}$. Furthermore, if $1 \in D$, then $\Cay({\rm{Dic}}(A,y),S)^D$ is connected.
\end{theorem}

\subsection{Proof of Theorem \ref{thm:main_2}} \label{ssec:main_2.2}

As in \cite{CFLLS20}, the proof aims at finding subsets $S_1^{D}, S_2^{D}$ of $A$ such that, with $S^{D} = S_1^{D} \cup x S_2^{D}$, the following conditions hold:

\noindent
$\bullet$ $1 \not \in S^{D}$, 

\noindent
$\bullet$ $(S^{D})^{-1} = S^{D}$, 

\noindent
$\bullet$ $S_1^{D}, S_2^{D} \in B(A)$, 

\noindent
$\bullet$ $\Cay({\rm{Dic}}(A,y),S)^D = \Cay({\rm{Dic}}(A,y),S^{D})$.

\noindent
 Corollary \ref{coro:main_1.2} will then yield that $\Cay({\rm{Dic}}(A,y),S)^D$ is integral and $S_2^D=(S_2^D)^{-1}$. As for the implication ``$1 \in D$ $\Rightarrow$ $\Cay({\rm{Dic}}(A,y),S)^D$ connected", the construction below will give that $S^D$ contains $S$ if $1 \in D$. Therefore, as $\langle S\rangle = {\rm{Dic}}(A,y)$, we will also have $\langle S^D \rangle = {\rm{Dic}}(A,y)$ if $1 \in D$.

First, we assume $D= \{d\}$ with $d \geq 1$ and, for simplicity, we set $\Gamma = \Cay({\rm{Dic}}(A,y),S)$.

\begin{lemma} \label{fold}
For every positive integer $n$, we have 
\begin{equation} \label{eq:fold}
\displaystyle{S^{(n)} = \bigcup_{\substack{0 \leq k, \ell \leq n \\ k+ \ell = n}} S_1^{(k)} (xS_2)^{(\ell)}.}
\end{equation}
\end{lemma}

\begin{proof}
We proceed by induction on $n$. If $n=1$, then \eqref{eq:fold} holds as $S = S_1 \cup xS_2$. Now, as\-su\-me \eqref{eq:fold} holds for a given $n \geq 1$, and let $x_1, \dots, x_n, x_{n+1} \in S$. Then there are integers $k,\ell \geq 0$ with $k + \ell = n$ and $x_1 \cdots x_n \in S_1^{(k)}(xS_2)^{(\ell)}$. If $x_{n+1} \in xS_2$, then $x_1 \cdots x_n x_{n+1} \in S_1^{(k)}(xS_2)^{(\ell+1)}$. Hence, assume $x_{n+1} \in S_1$. Since $x_1 \cdots x_n \in S_1^{(k)}(xS_2)^{(\ell)}$, there are $s_1, \dots, s_k \in S_1$ and $t_1, \dots, t_\ell \in S_2$ with
$x_1 \cdots x_n = s_1 \cdots s_k (xt_1) \cdots (xt_\ell)$. Then we have 
$$x_1 \cdots x_n x_{n+1} = \left\{
    \begin{array}{ll}
        s_1 \cdots s_k x_{n+1} (xt_1) \cdots (xt_\ell) & \mbox{if} \, \,  \ell \, \, \mbox{is\, even} \\
        s_1 \cdots s_k x_{n+1}^{-1} (xt_1) \cdots (xt_\ell) & \mbox{if} \, \,  \ell \, \,  \mbox{is\, odd.}
    \end{array}
\right.
$$
As $S_1^{-1} = S_1$, we get $x_1 \cdots x_n x_{n+1} \in S_1^{(k+1)}(xS_2)^{(\ell)}$ (in both cases).
\end{proof}

Now, set 
$$S^D = S^{(d)} \, \big\backslash \displaystyle{\bigcup_{0 \leq m \leq d-1} S^{(m)}.}$$
Then two distinct vertices $v, w$ of $\Gamma^D$ are adjacent if and only if $v^{-1}w \in S^D$. Since $S^{(0)}=\{1\}$, we have $1 \not \in S^{D}$ and, as $S^{-1} = S$, we have
$(S^{D})^{-1} = S^{D}$, thus showing $\Gamma^D = {\rm{Cay}}({\rm{Dic}}(A,y), S^{D})$. Moreover, if $d=1$, then $S^D=S$. Furthermore, Lemma \ref{fold} yields
$$\begin{array}{lll}
S^{D} & = & \displaystyle{S^{(d)} \setminus \bigcup_{0 \leq m \leq d-1} S^{(m)}} \\
& = & \displaystyle{\Bigg(\bigcup_{k+ \ell = d} S_1^{(k)} (xS_2)^{(\ell)} \Bigg) \big\backslash \Bigg(\bigcup_{0 \leq m_1 + m_2 \leq d-1} S_1^{(m_1)} (xS_2)^{(m_2)} \Bigg)}\\
 & = & \displaystyle{\bigcup_{k+ \ell = d} \Bigg( S_1^{(k)} (xS_2)^{(\ell)} \, \big\backslash \Bigg(\bigcup_{0 \leq m_1 + m_2 \leq d-1} S_1^{(m_1)} (xS_2)^{(m_2)} \Bigg) \Bigg)}\\
 & = & \displaystyle{\bigcup_{k+ \ell = d} \, \bigcap_{0 \leq m_1 + m_2 \leq d-1} \Bigg( S_1^{(k)} (xS_2)^{(\ell)} \, \big\backslash  S_1^{(m_1)} (xS_2)^{(m_2)}  \Bigg)} \\
 & = & S_1^{D} \cup x S_2^{D},
\end{array}$$
with
$$S_1^{D} = \bigcup_{\substack{k + \ell = d \\ \ell \, {\rm{even}}}} \, \bigcap_{0 \leq m_1 + m_2 \leq d-1} \Bigg( S_1^{(k)} (xS_2)^{(\ell)} \, \big\backslash  S_1^{(m_1)} (xS_2)^{(m_2)}  \Bigg) \subseteq A,$$
$$S_2^{D} = x^3 \bigcup_{\substack{k + \ell = d \\ \ell \, {\rm{odd}}}} \, \bigcap_{0 \leq m_1 + m_2 \leq d-1} \Bigg( S_1^{(k)} (xS_2)^{(\ell)} \, \big\backslash  S_1^{(m_1)} (xS_2)^{(m_2)}  \Bigg) \subseteq A.$$

To conclude the proof in the case $|D| = \{1\}$, it remains to show $S_1^{D}, S_2^{D} \in B(A)$. We start with $S_1^D$. To that end, fix $k, \ell$ with $k+ \ell=d$ and $\ell$ even, and fix $m_1, m_2$ with $0 \leq m_1 + m_2 \leq  d-1$. By the assumption, $S_1 \in B(A)$. Lemma \ref{lemma:ks} then yields $S_1^{(k)} \in B(A)$. Moreover, as $S_2 \in B(A)$, we have $S_2^{-1} = S_2$ (see Lemma \ref{lemma:ap_0}). Using $yS_2 = S_2$ and that $\ell$ is even, we then get $(xS_2)^{(\ell)} = S_2^{(\ell)}$. In particular, Lemma \ref{lemma:ks} yields $S_1^{(k)} (xS_2)^{(\ell)} \in B(A)$, and $S_1^{(m_1)} (xS_2)^{(m_2)} \in B(A)$ if $m_2$ is even. Hence, $S_1^{(k)} (xS_2)^{(\ell)} \, \big\backslash  S_1^{(m_1)} (xS_2)^{(m_2)} \in B(A)$ if $m_2$ is even. Next, if $m_2$ is odd, we have $S_1^{(k)} (xS_2)^{(\ell)} \subseteq A$ and $(S_1^{(m_1)} (xS_2)^{(m_2)}) \cap A = \emptyset$. Hence, $S_1^{(k)} (xS_2)^{(\ell)} \, \big\backslash  S_1^{(m_1)} (xS_2)^{(m_2)} =S_1^{(k)} S_2^{(\ell)}$ is also in $B(A)$, and we conclude that $S_1^D$ is in $B(A)$.

The proof is analogous for $S_2^D$. Namely, as $(S^D)^{-1} = S^D$, we have $yS_2^D = S_2^D$, that is, 
$$S_2^{D} =  \bigcup_{\substack{k + \ell = d \\ \ell \, {\rm{odd}}}} \, \bigcap_{0 \leq m_1 + m_2 \leq d-1} x\Bigg( S_1^{(k)} (xS_2)^{(\ell)} \, \big\backslash  S_1^{(m_1)} (xS_2)^{(m_2)}  \Bigg).$$
Fix $k, \ell$ with $k+ \ell=d$ and $\ell$ odd, and fix $m_1, m_2$ with $0 \leq m_1 + m_2 \leq  d-1$. Using the facts that $S_i^{-1} = S_i$ for $i=1,2$, $\ell$ is odd, and $yS_2 = S_2$, we have
$$x(S_1^{(k)} (xS_2)^{(\ell)} \setminus S_1^{(m_1)} (xS_2)^{(m_2)}) = \left\{
    \begin{array}{ll}
        S_1^{(k)} S_2^{(\ell)} & \mbox{if} \, \,  m_2 \, \, \mbox{is\, even} \\
         S_1^{(k)} S_2^{(\ell)} \, \big \backslash \, S_1^{(m_1)} S_2^{(m_2)} & \mbox{if} \, \,  m_2 \, \,  \mbox{is\, odd.}
    \end{array}
\right.
$$
As in the previous paragraph, we conclude that $S_2^D$ is in $B(A)$.

Finally, if $D = \{d_1, \dots, d_n\}$ with $n \geq 2$, set $S_i^D = S_i^{\{d_1\}} \cup \cdots \cup S_i^{\{d_n\}}$ for $i=1,2$, and $S^D = S_1^D \cup x S_2^D$. Then $S_1^{D}, S_2^{D} \subseteq A$, $1 \not \in S^{D}$, $(S^{D})^{-1} = S^{D}$, $S_1^{D}, S_2^{D} \in B(A)$, and $\Cay({\rm{Dic}}(A,y),S)^D = \Cay({\rm{Dic}}(A,y),S^{D})$. Moreover, we have $S \subseteq S^D$ if  $1 \in D$, thus concluding the proof.

\subsection{Corollaries} \label{ssec:main_2.3}

We start with the following result, which extends Corollary \ref{coro:main_1.2} and which is actually a combination of Corollary \ref{coro:main_1.2} and Theorem \ref{thm:main_2}:

\begin{corollary} \label{coro:main_2.1}
Let $S \subseteq {\rm{Dic}}(A,y)$ be such that $1 \notin S$, $S^{-1}=S$, and $\langle S\rangle = {\rm{Dic}}(A,y)$. Set $S= S_1 \cup xS_2$ with $S_1,S_2 \subseteq A$. The following three conditions are equivalent:

\vspace{0.5mm}

\noindent
{\rm{(1)}} $\Gamma=\Cay({\rm{Dic}}(A,y),S)$ is integral and $S_2^{-1} = S_2$, 

\vspace{0.5mm}

\noindent
{\rm{(2)}} given $D \subseteq \mathbb{N} \setminus \{0\}$ non-empty and finite, $\Gamma^D$ is an integral Cayley graph over ${\rm{Dic}}(A,y)$ and, setting $\Gamma^D = \Cay({\rm{Dic}}(A,y),S_1^D \cup xS_2^D)$ with $S_1^D, S_2^D \subseteq A$, we have $(S_2^D)^{-1} = S_2^D$,

\vspace{0.5mm}

\noindent
{\rm{(3)}} $S_1, S_2 \subseteq B(A)$.
\end{corollary}

We conclude this section with the next result, which yields infinitely many connected Cayley graphs over generalized dicyclic groups whose all distance powers are integral Cayley graphs, and of which a variant for distance integrality will be given in Corollary \ref{coro:main_4.1}:

\begin{corollary} \label{coro:main_2.2}
Let $S_1 \subseteq A$ be such that $1 \not \in S_1$, $S_1 \in B(A)$, and $\langle S_1 \rangle = A$ (e.g., $S_1 = A \setminus \{1\}$). Given a non-empty finite set $D$ of positive integers, the distance power $\Cay({\rm{Dic}}(A,y), S_1 \cup x\{1,y\})^D$ of the connected Cayley graph $\Cay({\rm{Dic}}(A,y), S_1 \cup x\{1,y\})$ is an integral Cayley graph over ${\rm{Dic}}(A,y)$. Moreover, setting $\Cay({\rm{Dic}}(A,y),S_1 \cup x\{1,y\})^D = \Cay({\rm{Dic}}(A,y),S_1^D \cup xS_2^D)$ with $S_1^D, S_2^D \subseteq A$, we have $(S_2^D)^{-1} = S_2^D$ and, if $1 \in D$, then $\Cay({\rm{Dic}}(A,y),S)^D$ is connected.
\end{corollary}

\begin{proof}
Set $S_ 2 = \{1,y\}$ and $S = S_1 \cup xS_2$.  Then $1 \not \in S$. Also, we have $yS_2 = S_2$ and, as $S_1 \in B(A)$, we have $S_1^{-1} = S_1$ (see Lemma \ref{lemma:ap_0}). Hence, $S = S^{-1}$, thus making $\Cay({\rm{Dic}}(A,y),S)$ well-defined. Moreover, $\langle S \rangle = {\rm{Dic}}(A,y)$, i.e., $\Cay({\rm{Dic}}(A,y),S)$ is connected. Furthermore, $S_1$ and $S_2$ are elements of $B(A)$. It then remains to apply Theorem \ref{thm:main_2} to conclude the proof.
\end{proof}

\begin{remark} \label{rk:2}
(1) Corollary \ref{coro:main_2.2} can be compared with \cite[Corollary 4.5]{CFLLS20}, which has the weaker assumption ``$|S_2|=2$" but the stronger ``$A$ cyclic". In fact, it is plausible that Corollary \ref{coro:main_2.2} holds under the weaker assumption ``$|S_2|=2$" (instead of ``$S_2 = \{1,y\}$"). Also, the authors prove there an analogue of their Corollary 4.5 where the assumption ``$|S_2|=2$" is replaced by ``$|S_2| = |A|-2$" (see \cite[Corollary 4.6]{CFLLS20}), still with $A$ cyclic. Once again, it is likely that such a result remains true for $A$ abelian. We leave the details to the interested reader. 

\vspace{1mm}

\noindent
(2) As already mentioned in the introduction, \cite[Theorem 3.4]{CFLLS20} provides sufficient conditions on a connected Cayley graph over a given dihedral group for its distance powers to be integral Cayley graphs. We leave to the interested reader the extension of this last result to generalized dihedral groups, as we did above in the dicyclic context.
\end{remark}

\section{Distance integrality} \label{sec:main_3}

\subsection{Main result} \label{ssec:main_3.1}

The next theorem, which characterizes distance integral Cayley graphs over generalized dicyclic groups, is the aim of this section. Recall that the function $\ell_S$ is defined in \eqref{eq:l_s}, and that, for a multi-set $(S,f)$ with $S$ finite and a function $g : S \rightarrow \mathbb{C}$, we set $g(S,f) = \sum_{s \in S} f(s) g(s)$. Also, the function $\ell_S(x \cdot)$ below is defined in \eqref{eq:l_s_point}.

\begin{theorem} \label{thm:main_3}
Let $S \subseteq {\rm{Dic}}(A,y)$ be such that $1 \not \in S$, $S^{-1} = S$, and $\langle S\rangle = {\rm{Dic}}(A,y)$. Then $\Cay({\rm{Dic}}(A,y),S)$ is distance integral if and only if $(A, \ell_S) \in C(A)$ and $|\pi(A, \ell_S(x \cdot))| \in \mathbb{Z}$ for every irreducible representation $\pi : A \rightarrow \mathbb{C}^*$ of $A$ with $A^2 \not \subseteq {\rm{ker}}(\pi)$.
\end{theorem}

\subsection{Proof of Theorem \ref{thm:main_3}} \label{ssec:main_3.2}

We first prove Theorem \ref{thm:main_3} under an extra assumption:

\begin{proposition} \label{prop:1}
Let $S \subseteq {\rm{Dic}}(A,y)$ be such that $1 \not \in S$, $S^{-1} = S$, and $\langle S\rangle = {\rm{Dic}}(A,y)$. Then $\Cay({\rm{Dic}}(A,y),S)$ is distance integral if and only if the following three conditions hold:

\vspace{0.5mm}

\noindent
{\rm{(1)}} $(A, \ell_S) \in C(A)$,

\vspace{0.5mm}

\noindent
{\rm{(2)}} $|\pi(A, \ell_S(x \cdot))| \in \mathbb{Z}$ for every irreducible representation $\pi : A \rightarrow \mathbb{C}^*$ of $A$ with $A^2 \not \subseteq {\rm{ker}}(\pi)$,

\vspace{0.5mm}

\noindent
{\rm{(3)}} $\pi(A, \ell_S(x \cdot ))=0$ for every irreducible representation $\pi : A \rightarrow \mathbb{C}^*$ of $A$ with $\pi(y) = -1$.
\end{proposition}

\begin{proof}[Proof of Proposition \ref{prop:1}]
The proof is analogous to that of \cite[Theorem 4.3]{HL21}. Namely, by Lemma \ref{lemma:HL}, it suffices to show that (1), (2), and (3) hold if and only if the eigenvalues of $\Phi_S(\rho)$ are all in $\mathbb{Z}$ for every irreducible representation $\rho$ of ${\rm{Dic}}(A,y)$, where $\Phi_S(\rho)$ is defined in \eqref{eq:Phi}. Below we use once more the classification from \S\ref{sssec:main_1.2.1}.

First, consider a one-dimensional representation $\rho : {\rm{Dic}}(A,y) \rightarrow \mathbb{C}^*$ of ${\rm{Dic}}(A,y)$. Then the unique entry of $\Phi_S(\rho)$ equals $\rho(G, \ell_S)$. As $A^2 \subseteq {\rm{ker}}(\rho)$ and $\rho(x) \in \{1,-1,i,-i\}$, and as ${\rm{Dic}}(A,y) = A \coprod xA$, we get that $\rho(G, \ell_S)$ is in $\mathbb{Z}$ if $\rho(x) = \pm 1$. However, if $\rho(x) = \pm i$, we have $\rho(G, \ell_S) = \rho(A, \ell_S) \pm i \rho(A, \ell_S(x \cdot))$ and, hence, $\Phi_S(\rho)$ is in $\mathbb{Z}$ if and only if $\rho(A, \ell_S(x \cdot))=0$.

Next, let $\pi : A \rightarrow \mathbb{C}^*$ be a one-dimensional representation of $A$ such that $A^2 \not \subseteq {\rm{ker}}(\pi)$, and let $R_\pi$ be two-dimensional representation of ${\rm{Dic}}(A,y)$ defined in \eqref{equmatrepind}. Setting $v_1 = (1, 0)^t$ and $v_2 = (0, 1)^t$, and letting $a$ denote an arbitary element of $A$, we have
$$\left\{\begin{array}{lll}
R_\pi(x)(v_1)  & = & v_2 \\
R_\pi(x)(v_2) & =& \pi(y) v_1 \\
R_\pi(a)(v_1) & = & \pi(a) v_1\\
R_\pi(a)(v_2) & = & \pi(a^{-1}) v_2
\end{array} \right..$$
Using \eqref{eq:varphi}, for $a \in A$, we then have
$$\left\{\begin{array}{lll}
\varphi_{1,1}^{R_\pi}(a) & = & \pi(a) \\
\varphi_{1,1}^{R_\pi}(xa) & = & 0 \\
\varphi_{1,2}^{R_\pi}(a) & = & 0\\
\varphi_{1,2}^{R_\pi}(xa) & = & \pi(a^{-1}) \pi(y) \\
\varphi_{2,1}^{R_\pi}(a) & = & 0 \\
\varphi_{2,1}^{R_\pi}(xa) & = & \pi(a) \\
\varphi_{2,2}^{R_\pi}(a) & = & \pi(a^{-1}) \\
\varphi_{2,2}^{R_\pi}(xa) & =& 0
\end{array} \right..$$
Hence,
$$\Phi_{S}(R_\pi) = {\begin{pmatrix} \pi(A, \ell_S) & \pi(y) \displaystyle{\sum_{a \in A} \ell_S(xa) \pi(a^{-1})}  \\  \pi(A, \ell_S(x \cdot)) & \displaystyle{\sum_{a \in A} \ell_S(a) \pi(a^{-1})} \end{pmatrix}} =  \begin{pmatrix} \pi(A, \ell_S) & \pi(y) \overline{\pi(A, \ell_S(x \cdot))}  \\  \pi(A, \ell_S(x \cdot)) & \displaystyle{\sum_{a \in A} \ell_S(a) \pi(a^{-1})} \end{pmatrix}.$$
Since $S = S^{-1}$, we have $\ell_S(a) = \ell_S(a^{-1})$ for $a \in A$ and, therefore,
$$\Phi_{S}(R_\pi) = \begin{pmatrix} \pi(A, \ell_S) & \pi(y) \overline{\pi(A, \ell_S(x \cdot))}  \\  \pi(A, \ell_S(x \cdot)) & \pi(A, \ell_S) \end{pmatrix}$$
We then have
\begin{align*}
{\rm det} \left( XI_2- \Phi_{S}(R_\pi) \right)&= \begin{vmatrix} X - \pi(A, \ell_S) & - \pi(y) \overline{\pi(A, \ell_S(x \cdot))}\\
- {\pi(A, \ell_S(x \cdot))} & X- \pi(A, \ell_S)
\end{vmatrix}\\
&=X^2-2\pi(A, \ell_S)X+ \pi(A, \ell_S)^2-\pi(y) |\pi(A, \ell_S(x \cdot))|^2.
\end{align*}
Hence, the eigenvalues of $\Phi_{S}(R_\pi)$ are $x_{1} = \pi(A, \ell_S) - \sqrt{\pi(y)} | \pi(A, \ell_S(x \cdot))|$ and $x_{2} = \pi(A, \ell_S) + \sqrt{\pi(y)} | \pi(A, \ell_S(x \cdot))|$, where $\sqrt{\pi(y)}$ is a fixed square root of $\pi(y)$ in $\mathbb{C}$. First, assume $\pi(y) = 1$. If $x_1, x_2$ are in $\mathbb{Z}$, then $\pi(A, \ell_S) = (x_1 + x_2)/2$ is in $\mathbb{Q}$. As $\pi(A, \ell_S)$ is also integral over $\mathbb{Z}$, we get $\pi(A, \ell_S) \in \mathbb{Z}$. Consequently, $|\pi(A, \ell_S(x \cdot))|$ is also in $\mathbb{Z}$. Conversely, if $\pi(A, \ell_S)$ and $|\pi(A, \ell_S(x \cdot))|$ are in $\mathbb{Z}$, then $x_1, x_2$ are also in $\mathbb{Z}$. Now, assume $\pi(y) = -1$. Then $x_1, x_2$ are in $\mathbb{Z}$ if and only if $\pi(A, \ell_S) \in \mathbb{Z}$ and $\pi(A, \ell_S(x \cdot))=0$. It then remains to use Lemma \ref{lemma:multi} to conclude.
\end{proof}

We now proceed to the proof of Theorem \ref{thm:main_3}. By Proposition \ref{prop:1}, it suffices to show that the third condition from the proposition is redundant. To that end, fix $a \in A$, and assume $xa  = x_1 \cdots x_r$ with $r \geq 1$ and $x_1, \dots, x_r \in S$. Set $S = S_1 \cup xS_2$ with $S_1, S_2 \subseteq A$. Since $xa \not \in A$, there exists $i \in \{1, \dots, r\}$ such that $x_i \in xS_2$. Fix such an $i$, assume $x_1 , \dots, x_{i-1} \in S_1$, and set $x_i =x s_i$ with $s_i \in S_2$. Then we have
$$x^3a = y (x_1 \cdots x_r) = (x_1 \cdots x_{i-1}) y x s_i (x_{i+1} \cdots x_r) = (x_1 \cdots x_{i-1}) (x(ys_i))(x_{i+1} \cdots x_r).$$
Since $S^{-1} = S$, we must have $yS_2 = S_2$ and, in particular, $ys_i \in S_2$. Consequently, we have $\ell_S(x^3a) \leq \ell_S(xa)$ and, by swapping $x^3a$ and $xa$, we actually have $\ell_S(x^3a) = \ell_S(xa)$ for $a \in A$. Hence, given a one-dimensional representation $\pi : A \rightarrow \mathbb{C}^*$ of $A$ with $\pi(y) = -1$, we have
$$\pi(A, \ell_S(x \cdot )) = \sum_{a \in A} \ell_S(xa) \pi(a) = \sum_{a \in A} \ell_S(x^3a) \pi(ya) = - \sum_{a \in A} \ell_S(xa) \pi(a) = - \pi(A, \ell_S(x \cdot )),$$
that is, $\pi(A, \ell_S(x \cdot ))=0$, as needed.

\subsection{Corollaries} \label{ssec:main_3.3}

We start with the following statement, which provides a convenient sufficient condition for ${\rm{Cay}}({\rm{Dic}}(A,y), S)$ to be distance integral:

\begin{corollary} \label{coro:main_3.1}
Let $S \subseteq {\rm{Dic}}(A,y)$ be such that $1 \not \in S$, $S^{-1} = S$, and $\langle S\rangle = {\rm{Dic}}(A,y)$. Then $\Cay({\rm{Dic}}(A,y),S)$ is distance integral if $(A, \ell_S), (A, \ell_S(x \cdot))  \in C(A)$.
\end{corollary}

\begin{proof}
As $(A, \ell_S(x \cdot)) \in C(A)$, we may apply Lemma \ref{lemma:multi} to get that the multi-set $(A, \ell_S(x \cdot))$ is integral, i.e., for every one-dimensional representation $\pi : A \rightarrow \mathbb{C}^*$ of $A$, we have $\pi(A, \ell_S(x \cdot )) \in \mathbb{Z}$. In particular, $|\pi(A, \ell_S(x \cdot ))| \in \mathbb{Z}$ for every one-dimensional representation $\pi : A \rightarrow \mathbb{C}^*$ of $A$ with $A^2 \not \subseteq {\rm{ker}}(\pi)$, and it then remains to apply Theorem \ref{thm:main_3} to conclude the proof.
\end{proof}

We conclude with the next result, which provides a practical situation in which the converse in Corollary \ref{coro:main_3.1} holds. The next corollary is an analogue for generalized dicyclic groups of \cite[Corollary 4.5]{HL21} (which holds over generalized dihedral groups) and generalizes \cite[Corollary 4.2]{HL21b} (which assumes $A$ is cyclic):

\begin{corollary} \label{coro:main_3.2}
Let $S \subseteq {\rm{Dic}}(A,y)$ be such that $1 \not \in S$, $S^{-1} = S$, and $\langle S\rangle = {\rm{Dic}}(A,y)$. Set $S = S_1 \cup x S_2$ with $S_1, S_2 \subseteq A$, and assume $S_2^{-1} = S_2$. Then $\Cay({\rm{Dic}}(A,y),S)$ is distance integral if and only if $(A, \ell_S), (A, \ell_S(x \cdot))  \in C(A)$.
\end{corollary}

\begin{proof}
We claim that, under the extra assumption $S^{-1}_2 = S_2$, we have $\ell_S(xa) = \ell_S(xa^{-1})$ for $a \in A$. Then, given a one-dimensional representation $\pi : A \rightarrow \mathbb{C}^*$ of $A$ with $A^2 \not \subseteq {\rm{ker}}(\pi)$,
$$\overline{\pi(A, \ell_S(x \cdot))} = \sum_{a \in A} \ell_S(xa) \pi(a^{-1}) = \sum_{a \in A} \ell_S(xa^{-1}) \pi(a) = \sum_{a \in A} \ell_S(xa) \pi(a) = \pi(A, \ell_S(x \cdot)),$$
and it then remains to apply Lemma \ref{lemma:multi} and Theorem \ref{thm:main_3} to conclude the proof.

We now proceed to the proof of the claim. Given $a \in A$, set $xa = x_1 \cdots x_r$ with $r \geq 1$ and $x_1 , \dots, x_r \in S$. Then $xa^{-1} = x^3 (xa) x = x^3 (x_1 \dots x_r) x = (x^3 x_1 x) \cdots (x^3 x_r x).$ Fix $i \in \{1, \dots, r\}$, and assume first $x_i \in S_1$. Then $x^3 x_i x = x_i^{-1}$, which is in $S_1$ as $S=S^{-1}$. Now, assume $x_i = xs_i$ with $s_i \in S_2$. Then $x^3 x_i x = s_i x = x s_i^{-1}$, which is in $xS_2$ as $S_2^{-1} = S_2$. Hence, $\ell_S(xa^{-1}) \leq \ell_S(xa)$ and, by swapping $a$ and $a^{-1}$, we actually have $\ell_S(xa) = \ell_S(xa^{-1})$, as needed.
\end{proof}

\section{Comparison between integrality and distance integrality} \label{sec:main_4}

The next theorem is the aim of this section. It provides a practical situation in which integrality and distance integrality of Cayley graphs over generalized dicyclic groups are equivalent.

\begin{theorem} \label{thm:main_4}
Let $S \subseteq {\rm{Dic}}(A,y)$ be such that $1 \not \in S$, $S^{-1} = S$, and $\langle S\rangle = {\rm{Dic}}(A,y)$. Set $S = S_1 \cup  xS_2$ with $S_1, S_2 \subseteq A$, and assume $S_2^{-1} = S_2$. Then ${\rm{Cay}}({\rm{Dic}}(A,y), S)$ is distance integral if and only if ${\rm{Cay}}({\rm{Dic}}(A,y), S)$ is integral.
\end{theorem}

\begin{proof}
The proof is analogous to that of \cite[Theorem 5.2]{HL21}. Namely, assume first that ${\rm{Cay}}({\rm{Dic}}(A,y), S)$ is distance integral. As $S_2^{-1} = S_2$, we may apply Corollary \ref{coro:main_3.2} to get that $(A, \ell_S)$ and $(A, \ell_S(x \cdot))$ are in $C(A)$. Considering $(A, \ell_S)$ first, there exist finitely many elements $a_1, \dots, a_n$ of $A$ and, for $i=1, \dots, n$, a constant function $f_i : [a_i] \rightarrow \mathbb{N}$ such that $(A, \ell_S) = \sum_{i=1}^n ([a_i], f_i)$. Letting $I$ be the subset of $\{1, \dots, n\}$ consisting of all elements $i$ with $f_i(a_i) = 1$, we then have $S_1 = \cup_{i \in I} [a_i]$ and, hence, $S_1 \in B(A)$. Similarly, using $(A, \ell_S(x \cdot)) \in C(A)$, we have $S_2 \in B(A)$, and it then remains to apply Corollary \ref{coro:main_1.2} to get that ${\rm{Cay}}({\rm{Dic}}(A,y), S)$ is integral.

Conversely, assume ${\rm{Cay}}({\rm{Dic}}(A,y), S)$ is integral. To get that ${\rm{Cay}}({\rm{Dic}}(A,y), S)$ is distance integral, it suffices, by Corollary \ref{coro:main_3.1}, to show that $(A, \ell_S)$ and $(A, \ell_S(x \cdot))$ are elements of $C(A)$. 

To that end, fix $b_1, b_2 \in A$ with $\langle b_1 \rangle = \langle b_2 \rangle$. Then $b_2 = b_1^k$ for some $k \geq 1$, which is coprime to the order of $b_1$, and, as the order of $b_1$ divides $|A|$, we may and will assume that $k$ is coprime to $|A|$. Assume $b_1 \in S^{(r)}$ with $r \geq 0$. Then, by Lemma \ref{fold}, there exist non-negative integers $m_1, m_2$ such that $m_1 + m_2 = r$ and $b_1 \in S_1^{(m_1)}(xS_2)^{(m_2)}$. As already seen in the proof of Theorem \ref{thm:main_2}, we have $(xS_2)^{(\ell)} = S_2^{(\ell)}$ if $\ell$ is even and, using $S_2^{-1} = S_2$, we have $(xS_2)^{(\ell)} = xS_2^{(\ell)}$ if $\ell$ is odd. Consequently, we have $S_1^{(m_1)}(xS_2)^{(m_2)} = S_1^{(m_1)}S_2^{(m_2)}$ if $m_2$ is even and, using $S_1 = S_1^{-1}$, we have $S_1^{(m_1)}(xS_2)^{(m_2)} = x S_1^{(m_1)}S_2^{(m_2)}$ if $m_2$ is odd. But $b_1$ is in $A$ and, hence, $m_2$ has to be even. Then we may set $b_1=s_1 \cdots s_{m_1}t_1 \cdots t_{m_2}$ with $s_1, \dots, s_{m_1} \in S_1$ and $t_1, \dots, t_{m_2} \in S_2$. In particular, we have $b_2 = b_1^k = s_1^k \cdots s_{m_1}^k t_1^k \cdots t_{m_2}^k.$

First, fix $i \in \{1, \dots, m_1\}$. As ${\rm{Cay}}({\rm{Dic}}(A,y), S)$ is integral, we may apply Theorem \ref{thm:main_1} to get that $S_1$ is in $B(A)$. Set $\widetilde{B}(A) = \{[a_1], \dots, [a_n]\}$. Then, as recalled in \S\ref{ssec:prelim_2}, there exists $J \subseteq \{1, \dots, n\}$ such that $S_1 = \cup_{j \in J} [a_j]$. In particular, there is $j \in J$ such that $s_i \in [a_j]$. Moreover, since $k$ is coprime to $|A|$, it also holds that $k$ is coprime to the order of $a_j$ and, hence, $s_i^k \in[a_j] \subseteq S_1$. Similarly, as ${\rm{Cay}}({\rm{Dic}}(A,y), S)$ is integral and $S_2 = S_2^{-1}$, we may apply Corollary \ref{coro:main_1.2} to get that $S_2$ is in $B(A)$ and, hence, $t_i^k \in S_2$ for $i \in \{1, \dots, m_2\}$. As a consequence, we get $b_2 \in S_1^{(m_1)}S_2^{(m_2)} = S_1^{(m_1)}(xS_2)^{(m_2)} \subseteq S^{(r)}$, thus showing $\ell_S(b_2) \leq \ell_S(b_1)$. By swapping $b_1$ and $b_2$, we actually have $\ell_S(b_2)= \ell_S(b_1)$. Hence, $(A, \ell_S) = \sum_{i=1}^n ([a_i], f_i)$, where $f_i: [a_i] \rightarrow \mathbb{N}$ is the constant function fulfilling $f_i(a_i) = \ell_S(a_i)$ for $i=1, \dots, n$. In particular, $(A, \ell_S) \in C(A)$.

Finally, fix $r \geq 0$ such that $xb_1 \in S^{(r)}$. Then, by Lemma \ref{fold}, there exist non-negative integers $m_1, m_2$ such that $m_1 + m_2 = r$ and $xb_1 \in S_1^{(m_1)}(xS_2)^{(m_2)}$ and, by the above, $m_2$ should be odd. Moreover, we have $xb_1 \in S_1^{(m_1)}(xS_2)^{(m_2)} = x S_1^{(m_1)}S_2^{(m_2)}$, i.e., $b_1 \in S_1^{(m_1)}S_2^{(m_2)}$. Hence, as above, we have $b_2 = b_1^k \in S_1^{(m_1)}S_2^{(m_2)}$, that is, $xb_2 \in x S_1^{(m_1)}S_2^{(m_2)} = S_1^{(m_1)}(xS_2)^{(m_2)} \subseteq S^{(r)}$, thus showing $\ell_S(xb_2) \leq \ell_S(xb_1)$. By swapping $b_1$ and $b_2$, we actually have $\ell_S(xb_2)= \ell_S(xb_1)$. Hence, $(A, \ell_S(x \cdot)) = \sum_{i=1}^n ([a_i], g_i)$, where $g_i: [a_i] \rightarrow \mathbb{N}$ is the constant function fulfilling $g_i(a_i) = \ell_S(xa_i)$ for $i=1, \dots, n$. In particular, $(A, \ell_S(x \cdot)) \in C(A)$.
\end{proof}

We conclude our paper with the next corollary, which provides infinitely many distance integral connected Cayley graphs over generalized dicyclic groups whose connected distance powers are distance integral Cayley graphs:

\begin{corollary} \label{coro:main_4.1}
Let $S_1 \subseteq A$ be such that $1 \not \in S_1$, $S_1 \in B(A)$, and $\langle S_1 \rangle = A$ (e.g., $S_1 = A \setminus \{1\}$). Then every connected distance power of $\Cay({\rm{Dic}}(A,y), S_1 \cup x\{1,y\})$ is a Cayley graph over ${\rm{Dic}}(A,y)$, which is distance integral.
\end{corollary}

\begin{proof}
Fix a non-empty finite set $D$ of positive integers such that $\Cay({\rm{Dic}}(A,y), S_1 \cup x\{1,y\})^D$ is connected. Then, by Corollary \ref{coro:main_2.2}, there exist subsets $S_1^D, S_2^D$ of $A$ such that $\Cay({\rm{Dic}}(A,y), S_1 \cup x\{1,y\})^D = \Cay({\rm{Dic}}(A,y), S_1^D \cup xS_2^D)$. Moreover, by the same corollary, $\Cay({\rm{Dic}}(A,y), S_1 \cup x\{1,y\})^D$ is integral and $S_2^D = (S_2^D)^{-1}$. As $\Cay({\rm{Dic}}(A,y), S_1 \cup x\{1,y\})^D$ is connected, we may then apply Theorem \ref{thm:main_4} to conclude that $\Cay({\rm{Dic}}(A,y), S_1 \cup x\{1,y\})^D$ is distance integral.
\end{proof}

\begin{remark}
Since there is the implication ``$1 \in D$ $\Rightarrow$ $\Cay({\rm{Dic}}(A,y), S_1 \cup x\{1,y\})^D$ connected" in Corollary \ref{coro:main_2.2}, we have the following more practical version of Corollary \ref{coro:main_4.1}:

\vspace{1mm}

\noindent
{\it{Let $S_1 \subseteq A$ be such that $1 \not \in S_1$, $S_1 \in B(A)$, and $\langle S_1 \rangle = A$. Let $D$ be a non-empty finite set of positive integers with $1 \in D$. Then the distance power $\Cay({\rm{Dic}}(A,y), S_1 \cup x\{1,y\})^D$ of $\Cay({\rm{Dic}}(A,y), S_1 \cup x\{1,y\})$ is a Cayley graph over ${\rm{Dic}}(A,y)$, which is distance integral.}}
\end{remark}

\bibliography{Biblio2}
\bibliographystyle{alpha}

\end{document}